\documentclass[english]{amsart}

\usepackage{amsmath,amssymb,enumerate}
\usepackage{amscd,amsxtra,calc}

\usepackage[T1]{fontenc}
\usepackage[all]{xy}

\usepackage{babel}
\usepackage{amstext}
\usepackage{amsmath}
\usepackage{amsfonts}
\usepackage{latexsym}
\usepackage{ifthen}
\usepackage{verbatim}
\usepackage{enumitem}
\usepackage{tikz-cd}

\usepackage{xypic}
\xyoption{all}
\newcommand{\Cl}{{\rm Cl}}

\newtheorem{lemma1}{}[section]

\newenvironment{lemma}{\begin{lemma1}{\bf Lemma.}}{\end{lemma1}}

\newenvironment{theorem}{\begin{lemma1}{\bf Theorem.}}{\end{lemma1}}

\newenvironment{corollary}{\begin{lemma1}{\bf Corollary.}}{\end{lemma1}}
\newenvironment{remark}{\begin{lemma1}{\bf Remark.}\rm}{\end{lemma1}}
\newenvironment{definition}{\begin{lemma1}{\bf Definition.}}{\end{lemma1}}

\newenvironment{remark*}{{\bf Remark.}}{}
\newenvironment{example*}{{\bf Example.}}{}

\newcommand{\R}{\ensuremath{\mathbb{R}}}
\newcommand{\Q}{\ensuremath{\mathbb{Q}}}
\newcommand{\Z}{\ensuremath{\mathbb{Z}}}

\newcommand{\N}{\ensuremath{\mathbb{N}}}
\newcommand{\PP}{\ensuremath{\mathbb{P}}}

\newcommand{\merom}[3]{\ensuremath{#1\colon #2 \dashrightarrow #3}}

\newcommand{\holom}[3]{\ensuremath{#1\colon #2  \rightarrow #3}}
\newcommand{\fibre}[2]{\ensuremath{#1^{-1} (#2)}}

\makeatletter
\ifnum\@ptsize=0  \fi \ifnum\@ptsize=2  \fi 

\newcommand\sF{{\mathcal F}}
\newcommand\sG{{\mathcal G}}

\newcommand\sO{{\mathcal O}}

\makeatletter
\@namedef{subjclassname@2020}{\textup{2020} Mathematics Subject Classification}
\makeatother

\newcommand{\pic}[0]{\operatorname{Pic}}

\newcommand{\To}{\longrightarrow}

\newcommand{\NE}[1]{ \overline { \mathrm{NE} }(#1)}

\usepackage{xcolor}

\title{Projectivity criteria for K\"ahler morphisms}
\date{\today}


\author{Beno\^it Claudon}
\author{Andreas H\"oring}

\address{Beno\^it Claudon\\Univ Rennes, CNRS, IRMAR - UMR 6625, F-35000 Rennes, France et Institut Universitaire de France}
\email{Benoit.Claudon@univ-rennes.fr}

\address{Andreas H\"oring\\Universit\'e C\^ote d'Azur, CNRS, LJAD, France}
\email{Andreas.Hoering@univ-cotedazur.fr}

\begin{document}

\begin{abstract} 
In this short note we prove two projectivity criteria for fibrations between mildly singular compact K\"ahler spaces. They are the relative versions of the celebrated criteria of Kodaira and Moishezon. As an application we obtain that the MRC fibration
always has a model that is a projective morphism.
\end{abstract}

\maketitle

\vspace{-0.5cm}

\section{Introduction}

Given a compact K\"ahler manifold $X$, one of the most basic questions is to decide
whether $X$ is (the analytification of) a projective manifold. There are two well-known cases where this always holds:

\begin{itemize}
\item Kodaira's criterion: $X$ is projective if $H^0(X, \Omega_X^2)=0$.
\item Moishezon's theorem: $X$ is projective if there exists a line bundle 
$L \rightarrow X$ that is big.
\end{itemize}

In view of the recent progress on the MMP for projective morphisms \cite{Fuj22a, DHP24}, it is interesting to find sufficient conditions for a K\"ahler morphism to be projective. 
In this short note we generalise the classical criteria to the relative situation. Our first result is: 

\begin{theorem}\label{th:projectivity criterion}
Let $f\colon X\to Y$ be a fibration between compact Kähler manifolds.
Assume one of the following:
\begin{itemize}
\item The natural  map
$f^*:H^0(Y,\Omega^2_Y)\To H^0(X,\Omega^2_X)$
is an isomorphism.
\item The morphism $f$ is Moishezon, i.e. there exists a line bundle $L \rightarrow X$
that is $f$-big.
\end{itemize}
Then $f$ is a projective morphism.
\end{theorem}

The first assumption gives a global version of \cite[Cor.1.2]{Bin83} \cite[Prop.3.3.1]{Nak02b} and is verified by every fibration that is cohomologically constant (cf. \cite{DS21} for the terminology). In particular it holds for any smooth model of the MRC fibration of a compact K\"ahler manifold, see below. 
The second statement improves \cite[Thm.10.1]{CP00} where it was shown that $f$ is locally projective. 
In the case where $Y$ is a point Boucksom's version  of Moishezon's theorem \cite[Thm.1.2.13]{Bou02} gives a more precise result: if $X$ is K\"ahler and Moishezon, then the projection of a K\"ahler class on the real Neron-Severi space remains K\"ahler. It is not clear to us if this still holds in the relative setting. 

The proofs of the two statements are quite different. For the first case we use the Hodge decomposition to construct a relatively ample line bundle, whereas in the second case we use a relative MMP  to show that the projectiveness of the fibration is preserved under suitably chosen bimeromorphic modifications. The latter technique also works for mildly singular spaces, allowing us to prove a variant of the main statement that is suitable for applications in the MMP, cf. Theorem~\ref{theorem-singular}. 
In particular we obtain:

\begin{theorem} \label{theorem-MRC}
Let $X$ be a normal compact K\"ahler space with klt singularities.
Then there exists a model of the MRC fibration that is a projective morphism.
\end{theorem}

This answer a question asked to us by Juanyong Wang.
In fact our proof shows that every holomorphic model $X' \rightarrow Y'$ with $X'$ and $Y'$ smooth is a projective morphism, cf. also Remark~\ref{remark-MRC-final}.

In Section \ref{section-applications} we give some applications of our main results, for example to bimeromorphic morphism between strongly $\Q$-factorial klt spaces.

\subsection*{Acknowledgements} We are very grateful to C. Voisin for a discussion that lead us to the proof
of the first case of Theorem~\ref{th:projectivity criterion}. We thank
F.~Campana, O.~Fujino, J.~Koll\'ar and T.~Peternell for some very helpful communications.

Both authors would like to thank Institut Universitaire de France for providing excellent working conditions. BC benefits from the support of the French government ``Investissements d’Avenir'' program integrated to France 2030, bearing the following reference ANR-11-LABX-0020-01.

\section{Basic facts and notation}

All complex spaces are supposed to be separated and of finite dimension,  
a complex manifold is a smooth irreducible complex space.
An analytic variety is a complex space that is irreducible and reduced.
A fibration is a proper surjective morphism with connected fibres between complex spaces.
We refer to \cite{Gra62, Fuj78, Dem85} for basic definitions about $(p,q)$-forms
and K\"ahler forms in the singular case.  

We use the standard terminology of the MMP as explained in \cite{KM98, Deb01},
cf. also \cite{Nak87} for foundational material in the case of projective morphisms.

\subsection{Relatively ample line bundles}\label{subs:proj morphisms}

Let $X$ be a normal compact complex space with at most rational singularities.
Suppose that $X$ is in the Fujiki class, i.e. $X$ is bimeromorphic to a compact K\"ahler manifold. 
A $(1,1)$-class on $X$ is an element of $N^1(X):=H^{1,1}_{\rm BC}(X)$,
the Bott-Chern group of $(1,1)$-currents that are locally $\partial \bar \partial$-exact (cf. \cite[Sect.2]{HP16} for details). 
Even if $X$ is projective, the inclusion 
$$
\mbox{NS}(X) \otimes \R \subset N^1(X)
$$
is typically not an equality. However if $H^2(X, \sO_X)=0$, we can apply Kodaira's criterion to a desingularisation
to see that $X$ is projective and every $(1,1)$-class is an $\R$-divisor class.

\begin{definition}
Let $\holom{f}{X}{Y}$ be a fibration between analytic varieties. The fibration $f$ is projective (resp. Moishezon) if there exists a line bundle $L \rightarrow X$ that
is relatively ample (resp. relatively big).

The fibration $f$ is locally projective (resp. locally Moishezon), if there exists a covering of $Y$ by open sets $U_i$ such that $\fibre{f}{U_i} \rightarrow U_i$ is projective (resp. Moishezon).
\end{definition}

\begin{remark} \label{remark-moishezon}
In the literature, e.g. \cite[VII, \S 6]{EMS7} a compact analytic variety is said to be Moishezon if its dimension is equal to its algebraic dimension, i.e. the transcendence degree of its field of meromorphic functions. If $X$ is smooth this is equivalent to the existence of a big line bundle \cite[VII, \S 6.3]{EMS7}.

Our definition of projective, resp. Moishezon morphism is more general than
in parts of the literature, e.g. \cite{Kol22}.
\end{remark}

\begin{remark} \label{remark-compose-projective}
The composition of two projective morphisms is in general not projective, but 
it is straightforward to show that this holds if the base of the fibrations are compact.
\end{remark}

The relative ampleness of a line bundle can be read off positivity properties of its Chern class:

\begin{lemma}\label{lem:relative positivity}
Let $f\colon X\to Y$ be a fibration between compact K\"ahler manifolds. A line bundle $L\in\pic(X)$ is $f$-ample if and only if there exists $[\omega_Y] \in H^2(Y,\R)$ a K\"ahler class such that $c_1(L)+f^*(\omega_Y)$  is a K\"ahler class on $X$.
\end{lemma}

We have not been able to locate this statement in the literature, cp. \cite[Sect.1]{FS90} for similar considerations. We thus describe a sketch of proof for the reader's convenience.

\begin{proof}[Sketch of proof]
If $L\in\pic(X)$ is $f$-ample, we can just apply \cite[Lemma 4.4]{Fuj78}.

Conversely, assume that  for some smooth metric $h$ on $L$ and
some K\"ahler form $\omega_Y$ on $Y$, the $c_1(L,h)+f^*(\omega_Y)$ form is Kähler.  We can then consider $U\subset Y$ a Stein open subset where $\left(\omega_Y\right)|_U=i\partial\overline{\partial} \varphi_U$ for some smooth strictly psh function $\varphi_U\colon U\to \R$. Over $X_U:=f^{-1}(U)$, the line bundle $L$ can be endowed with the metric $he^{-\varphi_U\circ f}$ that has positive curvature. Being proper over a Stein manifold, $X_U$ is holomorphically convex hence weakly pseudoconvex. We can then resort to the usual $L^2$ estimates over $X_U$ for the Hermitian line bundle $(L,he^{-\varphi_U\circ f})$ \cite{Dem82} and produce sections of $L^{\otimes m}$ defined over $X_U$ which separate points and tangent vectors (for some integer $m$ depending on $U$). The manifold $Y$ being compact, we cover it with finitely many open subsets as above and we then choose an $m$ uniformly such that $L^{\otimes m}$ is $f$-very ample.
\end{proof}

We recall the following application of Chow's lemma:

\begin{lemma} \label{lemma-moishezon-projective}
Let $\holom{f}{X}{Y}$ be a Moishezon fibration between compact analytic varieties.
Then there exists a projective log-resolution $\holom{\mu_0}{X_0}{X}$ such that
the composition $f \circ \mu_0$ is a projective morphism.
\end{lemma}

\begin{proof}
It is shown in \cite[Lem.2.18]{DH20} that we can find a log-resolution $\mu_0$
such that $f \circ \mu_0$ is a projective morphism. 
The property that $\mu_0$ is also projective is not claimed in \cite{DH20}, but can 
be achieved by applying the analytic version of Chow's lemma \cite[Cor.2]{Hir75} to the bimeromorphic map $\mu_0$.
\end{proof}

\subsection{Preliminaries on MMP}

A normal projective variety $X$ is $\Q$-factorial if every Weil divisor is $\Q$-Cartier, or more formally the natural map $\pic(X) \rightarrow \Cl(X)$ induces an isomorphism
$$
\pic(X) \otimes \Q \To \Cl(X) \otimes \Q.
$$
Following \cite{DH20} we extend the definition to the K\"ahler setting:

\begin{definition}
Let $X$ be a normal complex space. We denote by $W(X)$ the group of divisorial sheaves, i.e. the group of isomorphism classes of reflexive sheaves of rank one, endowed with the group operation $\sF \circ \sG := (\sF \otimes \sG)^{**}$.
\end{definition}

\begin{remark}
We have a natural inclusion 
$$
\left\{\begin{array}{rcl} \Cl(X) &\lhook\joinrel\To &W(X)\\
D & \longmapsto & \sO_X(D).
\end{array}\right.
$$
If $X$ is not the projective this inclusion is not necessarily surjective: let $S$ be a K3 surface that is very general
in its deformation space, so $\pic(S) = 0$. Let $X=\PP(\Omega_S)$, then 
we have $\pic(X) \simeq \Z K_X \neq 0$. However $X$ does not contain any divisor, cf. \cite[Cor.3.5]{AH21}, so $\Cl(X)=0$.
\end{remark}

\begin{definition} \label{definition-strongly-factorial} \cite[Defn.2.2(ix)]{DH20} 
A normal complex space  $X$ is strongly $\Q$-factorial if every divisorial sheaf
is $\Q$-Cartier, i.e. for every $\sF \in W(X)$ there exists $m \in \N$ such that
$\sF^{[m]}$ is locally free. 
\end{definition}

\begin{remark}
Since a line bundle defines a locally free sheaf, we have an inclusion
$\pic(X) \hookrightarrow W(X)$. The complex space is strongly $\Q$-factorial if the inclusion induces an isomorphism
$$
\pic(X) \otimes \Q \To W(X) \otimes \Q.
$$
In this case we have a well-defined morphism
$$
\left\{\begin{array}{rcl} W(X) & \To & N^1(X)\\
\sF & \longmapsto & c_1(\sF) := \frac{1}{m} c_1(\sF^{[m]})
\end{array}\right.
$$
where we choose $m \in \N^*$ such that $\sF^{[m]}$ is locally free. 
\end{remark}

Let us restate the results on MMP for projective morphisms in the form that we will
use in the sequel:

\begin{theorem} \cite{Nak87, DHP24, Fuj22a} \label{theorem-MMP}
Let $X$ be a normal compact strongly $\Q$-factorial K\"ahler space. Assume that there exists a boundary divisor $\Delta$ on $X$ such that the pair $(X, \Delta)$ is klt.
Let $\holom{f}{X}{Y}$ be a projective morphism onto a  normal compact K\"ahler space $Y$.

\begin{enumerate}[label={\rm(\arabic*)}]
\item If $K_X+\Delta$ is not $f$-nef, there exists a countable collection 
of rational curves $l_i \subset X$ such that $f(l_i)$ is a point and
$$
\NE{X/Y} = \NE{X/Y}_{K_X+\Delta \geq 0} + \sum_{i \in I} \R^+ [l_i].
$$
One has
$$
0 < -(K_X+\Delta) \cdot l_i \leq 2 \dim X
$$
and for every extremal ray $\R^+ [l_i] \subset \NE{X/Y}$ there exists a line bundle
$L \rightarrow X$ such that $\R^+ [l_i] = \NE{X/Y} \cap c_1(L)^\perp$.
\item\label{item:extremal ray} For every extremal ray $\R^+ [l_i] \subset \NE{X/Y}$ as above there exists a 
projective morphism $\holom{g}{Z}{Y}$ and a contraction morphism
$$
\holom{\varphi}{X}{Z}
$$
onto a normal compact K\"ahler space $Z$
such that $-(K_X+\Delta)$ is $\varphi$-ample and for any curve $C \subset X$ such that
$f(C)$ is a point we have
$$
\varphi(C)=pt. \ \Leftrightarrow \ [C] \in \R^+ [l_i].
$$
\item We can run a $K_X+\Delta$-MMP over $Y$, i.e. there exists a sequence of
bimeromorphic maps $\merom{\varphi_j}{X_j}{X_{j+1}}$ over $Y$ such that $\varphi_j$
is either the divisorial contraction of a $K_{X_j}+\Delta_j$-negative extremal ray
or its flip (if the contraction is small).
Moreover the pair $(X_{j+1}, (\varphi_j)_*(\Delta_j))$ is klt, the normal space $X_{j+1}$ is strongly $\Q$-factorial and the natural morphism $\holom{f_{j+1}}{X_{j+1}}{Y}$ is projective.
\item If $K_X+\Delta$ is $f$-pseudoeffective and $\Delta$ or $K_X+\Delta$ are $f$-big, any MMP with scaling by an $f$-ample line bundle terminates with a relative minimal model, i.e.
a projective morphism $\holom{f_m}{X_m}{Y}$ such that $K_{X_m}+\Delta_m$ is $f_m$-nef.
\item If $K_X+\Delta$ is not $f$-pseudoeffective, any MMP with scaling by an $f$-ample line bundle terminates with a relative Mori fibre space, i.e. 
a Mori fibre space $\holom{\varphi_m}{X_m}{Z}$ onto a normal compact K\"ahler space $Z$ 
of dimension at most $\dim X-1$ and a projective morphism $\holom{g}{Z}{Y}$ such that $f_m = g \circ \varphi_m$.
\end{enumerate}
\end{theorem}

\begin{proof}
The first statement is \cite[Thm.4.12(1)]{Nak87} (applied in the case $Y=W$), the statement on the length of the extremal ray is \cite[Thm.2.44(3)b)]{DHP24}, \cite[Thm.9.1]{Fuj22a}.

The existence of the contraction in the second statement is \cite[Thm.4.12(2)]{Nak87}
(applied in the case $Y=W$), note that any open neighbourhood $W \subset U \subset Y$ is equal to $Y$. While \cite{Nak87} does not state that the morphism $g$ is projective, this is clear from the first item: the supporting nef line bundle descends to $Z$ \cite[Thm.4.12(3,b)]{Nak87} and by construction is strictly positive on $\NE{Z/Y}$. Thus it is relatively ample by \cite[Prop.4.7]{Nak87}.
 It is clear that $Z$ is K\"ahler since $Y$ is K\"ahler by assumption and the morphism $g$ is K\"ahler (even projective).

The existence of MMP in the third statement is \cite[Thm.1.7]{Fuj22a}, \cite[Thm.1.4(1)]{DHP24}. Strong $\Q$-factoriality is shown in \cite[Lemma 2.5]{DH20}, the klt property is shown as in the projective case \cite[Cor.3.42, Cor.3.43]{KM98}

The two last statements are stated in  \cite[Thm.1.7]{Fuj22a}, \cite[Thm.1.4]{DHP24}
for a MMP with a scaling by a relatively ample {\em divisor}, but the arguments work
if we scale by the first Chern class of a relatively ample line bundle. Alternatively note that given an $f$-ample line bundle $L$, we can take a finite open cover $Y_k$ of the base such that each $Y_k$ satisfies the condition $(P)$ in \cite{DHP24, Fuj22a} and $c_1(L)|_{\fibre{f}{Y_k}}$
is represented by an $\R$-divisor.
\end{proof}

\begin{lemma} \label{lemma-q-factorialisation}
Let $X$ be a normal compact K\"ahler space, and let $\Delta$ be a boundary divisor
such that $(X, \Delta)$ is klt. Then there exists a small projective modification
$\holom{\nu}{X'}{X}$ such that $X'$ is strongly $\Q$-factorial.
\end{lemma}

\begin{proof}
This is shown for threefolds in \cite[Lem.2.27]{DH20}, by running a relative MMP
for some projective log-resolution. As shown by Theorem~\ref{theorem-MMP}, the results
of \cite{Fuj22a, DHP24} allow to run this MMP in any dimension. Thus the proof works without changes.
\end{proof}

We prove a variant of \cite[Lem.2.32]{DH20}, following their proof:

\begin{lemma} \label{lemma-factorise}
Let $\holom{\mu}{X'}{X}$ be a bimeromorphic morphism between normal compact K\"ahler spaces.
Assume that $X$ has strongly $\Q$-factorial klt singularities. Let 
$\holom{\tau}{X_0}{X'}$ be a log-resolution of $X'$ and the morphism $\mu$
such that $\mu_0:= \mu \circ \tau$ is a projective morphism (cf. Lemma~\ref{lemma-moishezon-projective}) with SNC exceptional locus. 

Then there exists a boundary divisor $\Delta_0$ such that $(X_0, \Delta_0)$ is klt and a decomposition of the morphism $\holom{\mu_0}{X_0}{X}$ into a finite sequence of
$K_{X_\bullet}+\Delta_\bullet$-negative bimeromorphic Mori contractions and flips
$$
\merom{\varphi_i}{X_i}{X_{i+1}} \qquad i=0, \ldots, m-1
$$
between normal compact K\"ahler spaces over $X$ such that $X_{m} \simeq X$.
\end{lemma}

\begin{proof}
We denote by $E_1, \ldots, E_k \subset X_0$ the $\mu_0$-exceptional divisors. 
Since $X$ has klt singularities we have
$$
K_{X_0} \sim_\Q \mu_0^* K_X + \sum_{i=1}^k a_i E_i
$$
with $a_i>-1$ for all $i=1,\ldots,k$. Thus we can choose $0 < \epsilon < 1$ such that $a_i-\epsilon>-1$ for
all $i=1,\ldots,k$. Since $X_0$ is smooth and the divisor $\sum_{i=1}^k E_i$ has SNC support, the pair
$(X_0, (1-\epsilon) \sum_{i=1}^k E_i)$ is klt. Since the morphism $\mu_0$ is projective and the boundary is big on the general fibre (an empty condition for bimeromorphic maps), we can run by Theorem~\ref{theorem-MMP} a terminating directed MMP 
$$
\merom{\varphi}{X_0}{X_m}
$$
over $X$ such that $K_{X_m}+(1-\epsilon) \sum_{i=1}^k  \varphi_* E_i$ is relatively nef for the bimeromorphic morphism $\holom{\mu_m}{X_m}{X}$. 

Since
$$
K_{X_m} + (1-\epsilon) \sum_{i=1}^k \varphi_*  E_i \sim_\Q \mu_m^* K_X + \sum_{i=1}^k (1-\epsilon + a_i) \varphi_* E_i,
$$
the effective divisor $\sum_{i=1}^k (1-\epsilon + a_i) \varphi_* E_i$ is $\mu_m$-nef and $\mu_m$-exceptional.
Thus the $\Q$-divisor $- \sum_{i=1}^k (1-\epsilon + a_i) \eta_* E_i$ is effective by the negativity lemma \cite[Lem.1.3]{Wan21} and
therefore the MMP $\varphi$ contracts all the divisors $E_i$. Since the MMP does not extract any divisors, the exceptional locus of $\mu_m$ has codimension at least two. By assumption $X$ is strongly $\Q$-factorial, so $\mu_m$ is an isomorphism \cite[Lem.2.4]{DH20}.
\end{proof}

\section{Proof of the main theorem}

We will prove the following singular variant, which obviously implies Theorem~\ref{th:projectivity criterion}.

\begin{theorem}\label{theorem-singular}
Let $f\colon X\to Y$ be a fibration between normal compact Kähler spaces.
Assume that $X$ has strongly $\Q$-factorial klt singularities.
Assume one of the following:
\begin{itemize}
\item The normal space $Y$ has klt singularities and the natural  map \cite[Thm.1.9]{KS21}
$f^*:H^0(Y,\Omega^{[2]}_Y)\To H^0(X,\Omega^{[2]}_X)$
is an isomorphism.
\item The morphism $f$ is Moishezon.
\end{itemize}
Then $f$ is a projective morphism.
\end{theorem}

Namikawa showed in \cite{Nam02} that a normal compact K\"ahler spaces with $1$-rational singularities that is Moishezon is even projective. Since klt singularities are rational this is more general than our statement (in the case where $Y$ is a point).
The technique is quite different: while Namikawa's proof aims at constructing an ample line bundle on $X$, we obtain the polarisation as a consequence of running a suitable MMP. 

\begin{proof}[Proof of Theorem \ref{theorem-singular}]
We divide the proof into three steps.

{\em Step 1. We prove the theorem under the first assumption, assuming moreover
that $X$ and $Y$ are smooth.}
Let us denote $d:=\dim X - \dim Y$, and let $\omega$ be a Kähler class on $X$.

Let $\eta\in H^2(X,\Q)$ be any class, then according to the assumption we can write
\begin{equation}
\label{equation-eta}
\eta=f^*(\gamma)+\beta+f^*(\bar\gamma)
\end{equation}
with $\gamma\in H^0(Y,\Omega^2_Y)$ and $\beta$ of type $(1,1)$.
In particular for every $y \in Y$ the restriction of $\eta$ to the fibre $X_y$
is of type $(1,1)$. Since $Y$ is compact we can choose  
$\eta\in H^2(X,\Q)$ sufficiently close to $\omega$ such that $\beta$ is still a relative K\"ahler class.

We claim that we can replace $\eta$ with $\widetilde{\eta}$ so that $\widetilde{\eta}-\eta\in f^*H^2(Y,\R)$ and $\widetilde{\eta}$ is also of type $(1,1)$. Then we know by the Lefschetz $(1,1)$-theorem that, up to replacing $\widetilde{\eta}$ by $m \widetilde{\eta}$ with $m \in \N$ sufficiently divisible, there exists a line bundle $L \in\pic(X)$ such that
$\widetilde{\eta} = c_1(L)$. By construction $c_1(L)$ is a
relative K\"ahler class, so by \cite[Lemma 4.4]{Fuj78} 
we can find a K\"ahler class $[\omega_Y]$ on $Y$ such that
$c_1(L)+f^* [\omega_Y]$ is K\"ahler. 
Therefore $L$ is $f$-ample by Lemma~\ref{lem:relative positivity}.

{\em Proof of the claim.}
Note that \eqref{equation-eta} implies that 
\begin{align}
f_*(\eta^d)&=f_*(\beta^d) \qquad\text{and}\nonumber\\
f_*(\eta^{d+1})&=f_*(\beta^{d+1})+(d+1) f_*(\eta^d)\left(\gamma+\bar\gamma\right). \label{eq:beta^(d+1)}
\end{align}
This is indeed a consequence of the projection formula
\[f_*\left(\alpha\wedge f^*(\delta)\right)=f_*(\alpha)\wedge \delta\]
and of the fact that $f_*$ is a morphism of Hodge structures of bidegree $(-d,-d)$. Let us note that the equality~\eqref{eq:beta^(d+1)} is nothing but the Hodge decomposition of $f_*(\eta^{d+1})$ in $H^2(Y,\R)$.

Since $f_*$ is defined over $\Q$ and $\eta$ is a rational class
we have $f_*(\eta^k) \in H^{2k-2d}(Y, \Q)$ for every $k \in \N$. Moreover 
\[
f_*(\eta^d)= \int_F \beta^d \in H^0(Y, \Q) 
\]
is a positive rational number, since by construction the restriction of $\beta$ to the general fibre $F$ is a K\"ahler class.
With this in mind, we can set
\[
\eta_Y :=\frac{1}{(d+1) f_*(\eta^d)}f_*(\eta^{d+1}) \in H^2(Y, \Q)
\]
and can consider
\[
\widetilde{\eta}:=\eta-f^*(\eta_Y)\in H^2(X,\Q).
\]
Using the equality~\eqref{eq:beta^(d+1)}, we see that
\[
\widetilde{\eta}=\beta-\frac{1}{(d+1) f_*(\eta^d)}f^*\left(f_*(\beta^{d+1})\right)
\]
is of type $(1,1)$.  This is the sought rational $(1,1)$-class and it proves the claim.

\medskip
{\em Step 2. We prove the theorem under the second assumption.} 
By Lemma~\ref{lemma-moishezon-projective} we can find a log-resolution
$\holom{\mu_0}{X_0}{X}$ such that $\mu_0$ and 
$$
f \circ \mu_0 : X_0 \rightarrow Y
$$ 
are projective morphisms. 
Applying Lemma~\ref{lemma-factorise} we obtain that the bimeromorphic map $\mu_0$
decomposes into a sequence
$$
\merom{\varphi_i}{X_i}{X_{i+1}}
$$
of divisorial contractions and flips over $X$. Denote by $\holom{\mu_i}{X_i}{X}$ the natural morphisms.
Since $f_0:=f \circ \mu_0$ is projective and $f_m \simeq f$
we are done if we show that projectiveness of $f_i = f \circ \mu_i$ is invariant under every step of our MMP. 
We will show this for the first contraction, the statement then follows by induction.

Denote by $\holom{\varphi}{X_0}{Z}$ the elementary Mori contraction 
of the extremal ray $\R^+ [C]$ in $\NE{X_0/X}$
(so $\varphi$
is small if $\varphi_0$ is a flip, and $\varphi=\varphi_0$ in the divisorial case), and let $\holom{g}{Z}{Y}$ be the natural map.
We have natural inclusions
$$
\NE{X_0/X} \subset \NE{X_0/Y} \subset \NE{X_0}
$$ 
and we claim that $\R^+ [C] \in  \NE{X_0/Y}$ is still an extremal ray. Indeed if $l_1, l_2
\in \NE{X_0/Y}$ are pseudoeffective classes such that $l_1+l_2 \in \R^+ [C]$, then
$$
0 = \varphi_* (l_1+l_2) = \varphi_* l_1 + \varphi_* l_2.
$$
Since $Z$ is K\"ahler by Theorem \ref{theorem-MMP}(2),
and $\varphi_* l_j \in \NE{Z}$ we have $\varphi_* l_j = 0$ for $j=1,2$.
Yet the kernel of 
$$
\holom{\varphi_*}{N_1(X_0)}{N_1(Z)}
$$
is exactly $\R [C]$. Thus we obtain $l_j \in \R^+ [C]$.

Since  $\R^+ [C] \in  \NE{X_0/Y}$ is an extremal ray 
we know by Theorem~\ref{theorem-MMP}~\ref{item:extremal ray}
that there exists a projective contraction morphism $\holom{\eta}{X_0}{\tilde X_1}$ such
that $\tilde X_1 \rightarrow Y$ is projective. 
Yet $\varphi_0$ and $\eta$ contract exactly the same curves,
so by the rigidity lemma \cite[Lem.4.1.13]{BS95} we have an isomorphism $\tilde X_1 \rightarrow Z$ that identifies $\varphi$ and $\eta$.

If $\varphi$ is divisorial we have $Z \simeq X_1$, so this proves the claim.
If $\varphi$ is small, note that the flip $\holom{\varphi^+}{X_1}{Z}$
is a projective morphism polarised by the $\Q$-line bundle $K_{X_1}$, so $g \circ \varphi^+$
is projective by Remark \ref{remark-compose-projective}.

\medskip
{\em Step 3. We prove the theorem under the first assumption.} 
Let $\holom{\mu_Y}{Y'}{Y}$ and $\holom{\mu_X}{X'}{X}$ be projective modifications
by compact K\"ahler manifolds such that we have an induced fibration
$f': X' \rightarrow Y'$. By \cite[Thm.1.2]{KS21} we have isomorphisms
$$
H^0(Y',\Omega^{2}_{Y'}) \simeq H^0(Y,\Omega^{[2]}_Y), \qquad
H^0(X',\Omega^{2}_{X'}) \simeq H^0(X,\Omega^{[2]}_X)
$$
and therefore the injection
$$
(f')^* : H^0(Y',\Omega^{2}_{Y'}) \rightarrow H^0(X',\Omega^{2}_{X'})
$$
is an isomorphism. By Step~1 the fibration $f'$ is projective, and therefore the 
fibration $f$ is Moishezon (note that $X'$ is strongly $\Q$-factorial, so the push-forward
of a relatively ample line bundle induces a relatively big line bundle). Yet by Step~2 this implies that $f$ is projective.
\end{proof}

\section{Applications of the main result}
\label{section-applications}

For lack of reference let us state the K\"ahler version of \cite[Thm.7.1]{Kol86}:

\begin{theorem} \cite{Kol86, Tak95} \label{kollar-takegoshi}
Let $\holom{f}{X}{Y}$ be a fibration between normal compact K\"ahler spaces with rational singularities,
and let $F$ be a general fibre. Then the following statements are equivalent:
\begin{itemize}
\item $R^i f_* \sO_X=0$ for all $i>0$;
\item $h^i(F, \sO_F)=0$ for all $i>0$.
\end{itemize}
\end{theorem}

\begin{proof}
Since $X$ has rational singularities we can replace it with a desingularisation without changing the statement. Now we follow the proof of \cite[Thm.7.1]{Kol86}: this proof is based on \cite[Thm.2.1]{Kol86} and general duality theory. The K\"ahler case of \cite[Thm.2.1]{Kol86} is shown in \cite[Thm.~II~and~IV]{Tak95}, duality theory in the analytic setting is established in \cite{RR70}.
\end{proof}

The assumption in Theorem~\ref{theorem-singular} can easily be verified: 

\begin{corollary}\label{cor:R1=R2=0}
Let $f\colon X\to Y$ be a fibration between normal compact Kähler spaces with klt singularities.
Assume that $X$ is strongly $\Q$-factorial.
If 
$$
R^1 f_*\sO_X = R^2 f_*\sO_X = 0,
$$
then $f$ is projective.
In particular 
\begin{itemize}
\item if $h^i(F, \sO_F)=0$ for all $i>0$,  or
\item if $f$ is bimeromorphic
\end{itemize}
the morphism $f$ is projective.
\end{corollary}

\begin{proof}
Since $R^1 f_*\sO_X = R^2 f_*\sO_X = 0$, the Leray spectral sequence shows that
$H^2(X, \sO_X) \simeq H^2(Y, \sO_Y)$. Since klt singularities are rational,
Hodge duality holds (see e.g. \cite[Cor.B.2.8]{Kir15}), so we obtain the isomorphism
in the assumption of Theorem~\ref{theorem-singular}. The last statement is now clear by
Theorem~\ref{kollar-takegoshi} (and since $f$ bimeromorphic corresponds to the case $\dim(F)=0$).
\end{proof}

\begin{remark}\label{rk:R2=0 implies locally proj}
It was proven by Nakayama \cite[Prop.3.3.1]{Nak02b} that a Kähler morphism $f\colon X\to Y$ with $R^2f_*\sO_X=0$ is locally projective. It is not possible to improve this statement, i.e. Corollary~\ref{cor:R1=R2=0} does not hold if we only assume $R^2f_*\sO_X=0$:
let $X$ be a compact complex torus of dimension $2$ and algebraic dimension $a(X)=1$.
Then the algebraic reduction is a fibration $f\colon X\to Y$ onto an elliptic curve $Y$.
For dimension reasons we have $R^2f_*\sO_X=0$ (but $R^1f_*\sO_X \simeq \sO_Y\neq 0$).
The morphism $f$ is not projective, since otherwise $X$ would also be projective.
\end{remark}

\begin{proof}[Proof of Theorem~\ref{theorem-MRC}]
The MRC fibration is an almost holomorphic fibration 
$\merom{f}{X}{Y}$ such that the general fibre $F$ is rationally chain connected.
Since $F$ has klt singularities, a desingularisation $F' \rightarrow F$ 
is rationally chain connected \cite{HM07} and therefore
$h^i(F, \sO_F) = h^i(F', \sO_{F'})=0$ for all $i>0$. Choose now any modification
$Y' \rightarrow Y$ such that $Y'$ is a compact K\"ahler space with klt singularities, and choose 
a resolution of the indeterminacies $X' \rightarrow X$ such that $X'$ is a compact K\"ahler space with strongly $\Q$-factorial klt singularities (e.g. $X'$ is smooth).
Then the morphism $f':X' \rightarrow Y'$ satisfies the assumptions of Corollary~\ref{cor:R1=R2=0} and is therefore projective. 
\end{proof}

\begin{remark} \label{remark-MRC-final}
In general the MRC fibration is not holomorphic, so a priori it has no distinguished bimeromorphic model. However we can choose a model which is canonical in many ways and a projective morphism: let $X$ be a normal compact K\"ahler space with klt singularities, and let
$F$ be a general fibre of the MRC fibration. Since the MRC fibration is almost holomorphic, the class of $F$ is contained in a unique irreducible
component of the cycle space $\mathcal C(X)$, and we denote its normalisation by $Z$.
The universal family $\Gamma \rightarrow Z$ defines a canonically defined meromorphic MRC fibration $\merom{f}{X}{Z}$. 
Denote by $Z_c \rightarrow Z$ the canonical modification (this exists by \cite[Thm.1.16]{Fuj22a}). Let $X_c$
be the canonical modification of the unique component of $\Gamma \times_Z Z_c$
that dominates $Z_c$. By Lemma~\ref{lemma-q-factorialisation} we can assume after 
a small modification that $X_c$ is strongly $\Q$-factorial. Thus we can apply 
Corollary~\ref{cor:R1=R2=0} to obtain that
$$
f_c\colon X_c\To Z_c
$$
is a projective morphism.
\end{remark}


\end{document}